\documentclass[11pt]{article}

\usepackage{amsmath, amsfonts, amsthm, graphicx,amssymb}
\usepackage[top=21mm, bottom=22mm, left=21mm, right=21mm]{geometry}
\usepackage{comment}
\usepackage{hyperref}
\hypersetup{
	colorlinks=true,
	linkcolor=blue,
	filecolor=magenta,      
	urlcolor=cyan,
	citecolor=blue
}
\usepackage{authblk}
\usepackage{enumitem}
\usepackage[none]{hyphenat}
\usepackage{tikz, tikz-3dplot}
\usetikzlibrary{calc, patterns,intersections, scopes}
\usetikzlibrary{positioning,arrows,shapes,decorations.markings,decorations.pathreplacing, decorations.pathmorphing,matrix,patterns}
\tikzstyle{vtx}=[circle,draw=black,fill=black,inner sep=0,minimum size=5pt,text=white,font=\footnotesize]

\newtheorem{theorem}{Theorem}

\newtheorem{lemma}[theorem]{Lemma}

\usepackage{cleveref}
\crefname{lemma}{lemma}{lemmas}
\Crefname{lemma}{Lemma}{Lemmas}
\crefname{claim}{claim}{claims}
\Crefname{claim}{Claim}{Claims}
\crefname{question}{question}{questions}
\Crefname{question}{Question}{Questions}

\theoremstyle{definition}

\DeclareMathOperator{\bx}{box}

\title{Box-Delaunay graphs of large chromatic number}
\author{Istv\'an Tomon}
\affil{Ume\r{a} University}
\affil{\texttt{istvan.tomon@umu.se}}
\date{}

\begin{document}

\maketitle
\sloppy

\begin{abstract}
   For every $n$, we construct a 2-dimensional $n$-element poset whose Hasse diagram has independence number $n\exp(-\Omega(\sqrt{\log n}))$, and consequently, chromatic number $\exp(\Omega(\sqrt{\log n}))$. This also yields an $n$-point planar set whose box-Delaunay graph satisfies the same bounds.
\end{abstract}

\section{Introduction}
Given a finite set of points $P$ in the plane in general position, the \emph{Delaunay graph} of $P$ is the graph on vertex set $P$, where $x$ and $y$ are joined by an edge if there exists a disk containing $x$ and $y$, but no other element of $P$. The importance of Delaunay graphs stems from the fact that they are planar, defining the \emph{Delaunay triangulation} of $P$, which is the dual of the Dirichlet-Voronoi diagram of $P$ (see e.g. \cite{BKOS}). Moreover, the Delaunay graph is closely related to conflict-free colorings. Here, a \emph{conflict-free} coloring of $P$ is a coloring such that every disk intersecting $P$ contains a point of unique color. Such colorings are motivated by frequency assignment problems in communication networks. Clearly, in a conflict-free coloring, every color class must be an independent set in the Delaunay graph. On the other hand, Even, Lotker, Ron, and Smorodinsky \cite{ELRS} proved that every set of $n$ points has a conflict-free coloring with $O(\log n)$ colors, building on the fact that the Delaunay graph is planar, thus contains an independent set of size $\Omega(n)$. 

Delaunay graphs and conflict-free colorings with respect to other geometric objects are also extensively studied. Of particular interest are \emph{box-Delaunay graphs}, also known as \emph{rectangle-visibility} graphs. Given a finite set of points $P\subset \mathbb{R}^2$, the \emph{box-Delaunay graph} of $P$, denoted by $D_{\bx}(P)$, is the graph on vertex set $P$ in which $x$ and $y$ are joined by an edge if there exists an axis-parallel rectangle containing $x$ and $y$, but no other element of $P$. Motivated by conflict-free colorings with respect to axis-parallel rectangles,  Even, Lotker, Ron, and Smorodinsky \cite{ELRS} and Har-Peled and Smorodinsky \cite{HS} asked whether every $n$-vertex box-Delaunay graph contains an independent set of size $\Omega(n)$. This was disproved by Chen, Pach, Szegedy, and Tardos \cite{CPST}, who showed that a random $n$-point set $P\subset [0,1]^2$ satisfies $\alpha(D_{\bx}(P))\leq O(n(\log \log n)^2/\log n)$ with high probability, where $\alpha(.)$ denotes the independence number. Recently, Jin, Kwan and Lichev \cite{JKL} sharpened this to  $\alpha(D_{\bx}(P))=\Theta(n(\log \log n)/\log n)$ with high probability.  On the other hand, Ajwani, Elbassioni, Govindarajan, and Ray \cite{AEGR} proved that every $n$-vertex box-Delaunay graph contains an independent set of size at least $\Omega(n^{0.617})$. It remains an intriguing problem (see e.g. \cite{CPST}) whether every $n$-vertex box-Delaunay graph contains an independent set of size $n^{1-o(1)}$. Box-Delaunay graphs of higher dimensional point sets are also studied \cite{JKL,Tom24}, where similar gaps persist.

In this paper, we construct planar point sets whose box-Delaunay graph has independence number substantially smaller than the previously best known $n(\log n)^{-1+o(1)}$ bound coming from  random point sets. This also implies a lower bound on the largest chromatic  numbers of box-Delaunay graphs.  Given an $n$-vertex graph $G$, we denote by $\chi(G)$ its chromatic number. Then we have the following simple inequality: $\chi(G)\geq n/\alpha(G)$.

\begin{theorem}\label{thm:main1}
    For $n\in \mathbb{Z}^+$, there exists an $n$-element set in $\mathbb{R}^2$ with box-Delaunay graph $G$ satisfying $$\alpha(G)\leq n2^{-c\sqrt{\log n}}\quad\text{and}\quad\chi(G)\geq 2^{c\sqrt{\log n}}\quad (c>0\text{ is an absolute constant}).$$
\end{theorem}

Box-Delaunay graphs are closely related to Hasse diagrams. Given a partially ordered set $(P,\prec)$, its \emph{Hasse diagram} is the graph on vertex set $P$, in which a comparable pair of elements $x\prec y$ is joined by an edge if there is no $z\in P$ such that $x\prec z\prec y$. In case $P$ is a set of points in the plane, one can define two partial orders $\prec_1$ and $\prec_2$ on $P$: if $(x_1,y_1),(x_2,y_2)\in P$, then $(x_1,y_1)\preceq_1 (x_2,y_2)$ if $x_1\leq x_2$ and $y_1\leq y_2$, and $(x_1,y_1)\preceq_2 (x_2,y_2)$ if $x_1\leq x_2$ and $y_2\leq y_1$. A poset isomorphic to either $(P,\prec_1)$ or $(P,\prec_2)$ for some set of points $P$ is called a \emph{2-dimensional} poset. The box-Delaunay graph $D_{\bx}(P)$ is the union of the Hasse diagrams of $(P,\prec_1)$ and $(P,\prec_2)$. In particular, this implies that the independence number of the Hasse diagram of $(P,\prec_1)$ is an upper bound on $\alpha(D_{\bx}(P))$. 

The chromatic number of Hasse diagrams has been studied since the 1970's. This question is closely tied to the independence number. It was first proved by Bollob\'as  \cite{Bollobas} that there exists an infinite family of posets (more precisely lattices) whose chromatic number tends to infinity with the number of vertices $n$. An alternative construction was provided later by Ne\v{s}et\v{r}il and R\"odl \cite{NR}. A stronger result, due to Brightwell and Ne\v{s}et\v{r}il \cite{BN}, states that there are $n$-element posets whose Hasse diagrams have independence number $o(n)$. In 2021, Suk and Tomon \cite{SukTom21} constructed $n$-vertex Hasse diagrams with chromatic number $\Omega(n^{1/4})$ and independence number $O(n^{3/4})$.

In the case of 2-dimensional posets,  K\v{r}\'i\v{z} and Ne\v{s}et\v{r}il \cite{KN} constructed a family with chromatic number tending to infinity. See also Dam\'asdi \cite{Damasdi} for an alternative construction. Matou\v{s}ek and P\v{r}\'iv\v{e}tiv\'y \cite{MP} asked whether there also exist 2-dimensional $n$-element posets whose Hasse diagram has independence number $o(n)$. This was answered affirmatively by Chen, Pach, Szegedy, and Tardos \cite{CPST}, who proved that the Hasse diagram of the 2-dimensional poset generated by a random point set has independence number $O(n(\log \log n)^2/\log n)$ with high probability, which was sharpened to $\Theta(n(\log \log n)/\log n)$ in \cite{JKL}. However, it remains open whether 2-dimensional posets can achieve similar chromatic and independence numbers as the above mentioned Hasse diagrams constructed by Suk and Tomon \cite{SukTom21}. 

The main result of this paper is a construction of a permutation $\pi\in S_n$ such that the Hasse diagram of the  2-dimensional poset induced by $P_{\pi}=\{(i,\pi(i)):i\in [n]\}$ (which we simply refer to as the Hasse diagram of the permutation $\pi$)  has fairly small independence number, and thus fairly large chromatic number. The following theorem also immediately implies \Cref{thm:main1} by recalling that the Hasse diagram of $P_{\pi}$ is a subgraph of $D_{\bx}(P_{\pi})$.

\begin{theorem}\label{thm:main}
    For $n\in \mathbb{Z}^+$, there exists an $n$-element permutation, whose Hasse diagram $G$ satisfies $$\alpha(G)\leq n2^{-c\sqrt{\log n}}\quad\text{and}\quad\chi(G)\geq 2^{c\sqrt{\log n}}\quad (c>0\text{ is an absolute constant}).$$
\end{theorem}

\section{Permutations with large chromatic number}

 Slightly unconventionally, we write $[n]=\{0,\dots,n-1\}$. In this section, we prove \Cref{thm:main}. In particular, we show the existence of a permutation $\pi$ of $[n]$, whose Hasse diagram has independence number at most $n2^{-\Omega(\sqrt{\log n})}$, as then this implies the appropriate bound on the chromatic number as well. We recall that the Hasse diagram of a permutation $\pi$ on $[n]$ is the graph on vertex set $[n]$, in which vertices $i,j\in [n]$ with $i<j$ are connected by an edge if $\pi(i)<\pi(j)$ and there exists no $\ell$ such that $i<\ell<j$ and $\pi(i)<\pi(\ell)<\pi(j).$
 
 In order to prove our desired result, we only consider values of $n$ which can be written as $n=2^{k^2/2}$ for some even integer $k$. Then the general result also follows: let $n_0\leq n$ be the largest integer of the form $2^{k^2/2}$, and let $\rho_0$ be a permutation on $n_0$ elements, whose Hasse diagram has independence number at most $n_02^{-\Omega(\sqrt{\log n_0})}$. Let $\pi$ be the permutation we get by taking the direct sum of $N=\lceil n/n_0\rceil$ copies of $\rho_0$, and deleting the final $n_0N-n$  elements to get a permutation of size exactly $n$. Then the independence number of the Hasse diagram of $\pi$ is at most $\lceil n/n_0\rceil\cdot n_02^{-\Omega(\sqrt{\log n_0})}=n2^{-\Omega(\sqrt{\log n})}$.

In the rest of the proof, we fix an even positive integer $k$, set $$m=2^{k/2}\quad\text{and}\quad n=m^k=2^{k^2/2}.$$ The idea of our construction is as follows. First, we define a ''noise gadget'', sampled from a simple probability distribution on permutations with bipartite Hasse diagrams. Then, we define a sequence of functions $\pi_0,\dots,\pi_k:[n]\rightarrow [n]$ as follows. We start with  $\pi_0\equiv 0$. Then, if $\pi_{j-1}$ is already defined for $j\geq 1$, we construct $\pi_{j}$ by perturbing the level sets of $\pi_{j-1}$ using independent samples of the noise gadget. Our final function $\pi=\pi_k$ will be the desired permutation of $[n]$ with high probability. We start by defining our noise gadget.

\medskip

\noindent
\textbf{Noise gadget.} Let $\alpha$ be the following random permutation of $[m]$. Partition $[m]$ randomly into two sets $A$ and $B$ such that each element of $[m]$ belongs to either $A$ or $B$ independently with probability $1/2$. Then $\alpha$ is the unique permutation which satisfies that $\tau(a)<\tau(b)$ for every $(a,b)\in A\times B$, and $\tau|_A$ and $\tau|_B$ are both monotone decreasing. See \Cref{fig:1} for an illustration.
\begin{figure}[ht]
\begin{center}
\begin{tikzpicture}[scale=0.4]

  \node[circle, draw, fill=red, inner sep=2pt] (v1) at (0,4) {};
  \node[circle, draw, fill=blue, inner sep=2pt] (v2) at (1,9) {};
  \node[circle, draw, fill=blue, inner sep=2pt] (v3) at (2,8) {};
  \node[circle, draw, fill=red, inner sep=2pt] (v4) at (3,3) {};
  \node[circle, draw, fill=red, inner sep=2pt] (v5) at (4,2) {};
  \node[circle, draw, fill=blue, inner sep=2pt] (v6) at (5,7) {};
  \node[circle, draw, fill=blue, inner sep=2pt] (v7) at (6,6) {};
  \node[circle, draw, fill=red, inner sep=2pt] (v8) at (7,1) {};
  \node[circle, draw, fill=blue, inner sep=2pt] (v9) at (8,5) {};
  \node[circle, draw, fill=red, inner sep=2pt] (v10) at (9,0) {};
  \draw (v1) -- (v2);
  \draw (v1) -- (v3);
  \draw (v1) -- (v6);
  \draw (v1) -- (v7);
  \draw (v1) -- (v9);
  \draw (v4) -- (v6);
  \draw (v4) -- (v7);
  \draw (v4) -- (v9);
  \draw (v5) -- (v6);
  \draw (v5) -- (v7);
  \draw (v5) -- (v9);
  \draw (v8) -- (v9);
\end{tikzpicture}
\end{center}
\caption{An example of a noise gadget and its Hasse diagram. The permutation $(4983276150)$ is a possible outcome for $m=10$.  The red points are $A=\{0,3,4,7,9\}$, the blue points are $B=\{1,2,5,6,8\}$.}
\label{fig:1}
\end{figure}
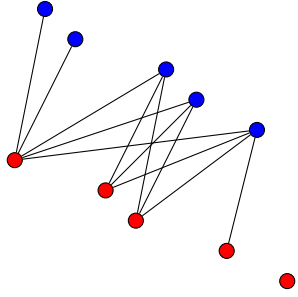

 The main advantage of the noise gadget is that every subset of $[m]$ has a small probability of being an independent set in the Hasse diagram.

\begin{lemma}\label{lemma:gadget}
For $I\subset [m]$, the probability that $I$ is independent in the Hasse diagram of $\alpha$ is  $(|I|+1)2^{-|I|}$.
\end{lemma}

\begin{proof}
Let $G$ be the Hasse diagram of $\alpha$, and let $(A,B)$ be the corresponding partition of $[m]$. For $a,b\in [m]$, $a$ and $b$ are joined by an edge in $G$ if $a\in A$, $b\in B$ and $a<b$. Therefore, $I$ is an independent set if and only if either $I\cap A=\emptyset$, or $I\cap B=\emptyset$, or every element of $I\cap B$ is smaller than every element of $I\cap A$. The probability of this is exactly $(|I|+1)2^{-|I|}$.
\end{proof}

 For $j\in [k+1]$ and $a\in [m^{k-j}]$, define the interval
$$J_{j,a}:=\{am^j+b:b\in [m^j]\}.$$ 
Note that $J_{j,0},\dots,J_{j,m^{k-j}-1}$ form a partition of $[n]$ into intervals of size $m^j$. Next, we define a sequence of random functions $$\pi_0,\dots,\pi_k:[n]\rightarrow [n]$$ such that for $j\in [k+1]$, $\pi_j$ has the following properties:
\begin{itemize}
    \item[(a)] if $a\in [m^{k-j}]$, then $\pi_j|_{J_{j,a}}$ is a permutation of $[m^j]$,
    \vspace{-5pt}\item[(b)] for every $i\in [j]$ and $a\in [m^{k-i}]$, $\pi_{i}|_{J_{i,a}}$ and $\pi_j|_{J_{i,a}}$ are isomorphic as permutations.
\end{itemize}
We start by setting $\pi_0\equiv 0$, then (a) is trivially satisfied, while (b) is vacuous. Assume that $\pi_{j-1}$ is already defined for some $j\in \{1,\dots,k\}$, and $\pi_{j-1}$ satisfies (a) and  (b). Then we define $\pi_j$ as follows. Let $a\in [m^{k-j}]$, and consider $\pi_{j-1}|_{J_{j,a}}$. The interval $J_{j,a}$ is partitioned into the $m$ intervals $J_{j-1,ma},\dots,J_{j-1,ma+m-1}$. On each of these intervals, $\pi_{j-1}$ is a permutation of $[m^{j-1}]$ by (a). Therefore, for every $s\in [m^{j-1}]$, the set 
$$K^{(j)}_{a,s}=\{x\in J_{j,a}: \pi_{j-1}(x)=s\}$$ 
has size $m$, as $K_{a,s}^{(j)}$ intersects $J_{j-1,ma+i}$ in a single element for $i\in [m]$. We refer to the sets $K_{a,s}^{(j)}$ as a \emph{level-$j$ slice}, or simply \emph{slice}. Write the elements of $K:=K_{a,s}^{(j)}$ as $x_0<\dots<x_{m-1}$. We sample the permutation $\beta=\beta^{(j)}_{a,s}$ from the distribution of $\alpha$, independently from all other samples, and set  $$\pi_j(x_t):=m\cdot \pi_{j-1}(x_t)+\beta(t)=ms+\beta(t)\quad\text{for }t\in [m].$$
As $(a,s)$ ranges through $[m^{k-j}]\times [m^{j-1}]$, this defines $\pi_j$ for every element of $[n]$. Moreover, we note that $\pi_j|_K$ is isomorphic to $\beta$ as a permutation.

We argue that $\pi_{j}$ satisfies (a) and (b). In order to show (a), it is enough to prove that  for every $a\in [m^{k-j}]$ and $b\in [m^j]$, there is an $x\in J_{j,a}$ such that $\pi_j(x)=b$. Write $b=ms+r$ as $s\in [m^{j-1}]$ and $r\in [m]$. Then there is a unique $t\in [m]$ such that $\beta_{a,s}^{(j)}(t)=r$, so writing $x_t$ for the $(t+1)$-st smallest element of $K^{(j)}_{a,s}$, we have $\pi_j(x_t)=b$. In order to see that (b) is true, observe that $\pi_{j-1}(x)=\lfloor \pi_{j}(x)/m\rfloor$ holds for every $x\in [n]$. By induction, this implies $\pi_i(x)=\lfloor \pi_j(x)/m^{j-i}\rfloor$ for every $i<j$.

Let $\pi=\pi_k$, then $\pi$ is a permutation of $[n]$ by (a). Next, we study the Hasse diagram $G$ of $\pi$. The main observation is that if $K=K^{(j)}_{a,s}$ is a slice, then the subgraph of $G$ induced on $K$ is isomorphic to the Hasse diagram of $\beta^{(j)}_{a,s}$. Indeed, $K\subset J_{j,a}$, and on the interval $J_{j,a}$, $\pi$ and $\pi_{j}$ are isomorphic as permutations by (b). Therefore, $G[J_{j,a}]$ is equal to the Hasse diagram of $\pi_j|_{J_{j,a}}$. Moreover, every $x\in K$ satisfies that $ms\leq \pi_j(x)<m(s+1)$, and no $x\in J_{j,a}\setminus K$ falls into the interval $\{ms,\dots,m(s+1)-1\}$. Thus, all edges of the Hasse diagram of the permutation $\pi_j|_{K}$  are also edges of the Hasse diagram of $\pi_j|_{J_{j,a}}$, and thus of $G$. See \Cref{fig:2} for an illustration.

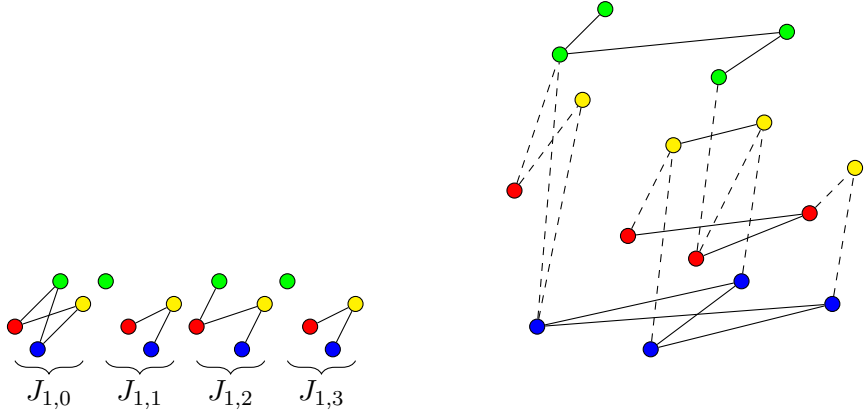
\begin{figure}[ht]
\begin{center}
\begin{tikzpicture}[scale=0.3]

\begin{scope}[shift={(-11,0.0)},scale=1.0]
  \node[circle, draw, fill=red, inner sep=2pt] (v1) at (0,1) {};
  \node[circle, draw, fill=blue, inner sep=2pt] (v2) at (1,0) {};
  \node[circle, draw, fill=green, inner sep=2pt] (v3) at (2,3) {};
  \node[circle, draw, fill=yellow, inner sep=2pt] (v4) at (3,2) {};
  \node[circle, draw, fill=green, inner sep=2pt] (v5) at (4,3) {};
  \node[circle, draw, fill=red, inner sep=2pt] (v6) at (5,1) {};
  \node[circle, draw, fill=blue, inner sep=2pt] (v7) at (6,0) {};
  \node[circle, draw, fill=yellow, inner sep=2pt] (v8) at (7,2) {};
  \node[circle, draw, fill=red, inner sep=2pt] (v9) at (8,1) {};
  \node[circle, draw, fill=green, inner sep=2pt] (v10) at (9,3) {};
   \node[circle, draw, fill=blue, inner sep=2pt] (v11) at (10,0) {};
  \node[circle, draw, fill=yellow, inner sep=2pt] (v12) at (11,2) {};
  \node[circle, draw, fill=green, inner sep=2pt] (v13) at (12,3) {};
  \node[circle, draw, fill=red, inner sep=2pt] (v14) at (13,1) {};
  \node[circle, draw, fill=blue, inner sep=2pt] (v15) at (14,0) {};
  \node[circle, draw, fill=yellow, inner sep=2pt] (v16) at (15,2) {};
  \draw (v1) -- (v3); \draw (v1) -- (v4) ; \draw (v2) -- (v3) ; \draw (v2) -- (v4);
  \draw (v6) -- (v8); \draw (v7) -- (v8) ; 
  \draw (v9) -- (v10); \draw (v9) -- (v12) ; \draw (v11) -- (v12) ; 
  \draw (v14) -- (v16); \draw (v15) -- (v16) ; 

  \node at (1.5,-2) {$J_{1,0}$}; \draw[decorate, decoration={brace, amplitude=5pt,mirror}]
    (0,-0.5) -- (3,-0.5);
  \node at (5.5,-2) {$J_{1,1}$};\draw[decorate, decoration={brace, amplitude=5pt,mirror}]
    (4,-0.5) -- (7,-0.5);
  \node at (9.5,-2) {$J_{1,2}$};\draw[decorate, decoration={brace, amplitude=5pt,mirror}]
    (8,-0.5) -- (11,-0.5);
  \node at (13.5,-2) {$J_{1,3}$};\draw[decorate, decoration={brace, amplitude=5pt,mirror}]
    (12,-0.5) -- (15,-0.5);
\end{scope}

\begin{scope}[shift={(11,0.0)},scale=1.0]
  \node[circle, draw, fill=red, inner sep=2pt] (v1) at (0,7) {};
  \node[circle, draw, fill=blue, inner sep=2pt] (v2) at (1,1) {};
  \node[circle, draw, fill=green, inner sep=2pt] (v3) at (2,13) {};
  \node[circle, draw, fill=yellow, inner sep=2pt] (v4) at (3,11) {};
  \node[circle, draw, fill=green, inner sep=2pt] (v5) at (4,15) {};
  \node[circle, draw, fill=red, inner sep=2pt] (v6) at (5,5) {};
  \node[circle, draw, fill=blue, inner sep=2pt] (v7) at (6,0) {};
  \node[circle, draw, fill=yellow, inner sep=2pt] (v8) at (7,9) {};
  \node[circle, draw, fill=red, inner sep=2pt] (v9) at (8,4) {};
  \node[circle, draw, fill=green, inner sep=2pt] (v10) at (9,12) {};
   \node[circle, draw, fill=blue, inner sep=2pt] (v11) at (10,3) {};
  \node[circle, draw, fill=yellow, inner sep=2pt] (v12) at (11,10) {};
  \node[circle, draw, fill=green, inner sep=2pt] (v13) at (12,14) {};
  \node[circle, draw, fill=red, inner sep=2pt] (v14) at (13,6) {};
  \node[circle, draw, fill=blue, inner sep=2pt] (v15) at (14,2) {};
  \node[circle, draw, fill=yellow, inner sep=2pt] (v16) at (15,8) {};

  \draw (v6) -- (v14); \draw (v9) -- (v14);
  \draw (v2) -- (v11); \draw (v2) -- (v15) ; \draw (v7) -- (v11); \draw (v7) -- (v15);
  \draw (v8) -- (v12);
  \draw (v3) -- (v5); \draw (v3) -- (v13) ; \draw (v10) -- (v13);

  \draw[dashed] (v1) -- (v3); \draw[dashed] (v1) -- (v4) ; \draw[dashed] (v2) -- (v3) ; \draw[dashed] (v2) -- (v4);
  \draw[dashed] (v6) -- (v8); \draw[dashed] (v7) -- (v8) ; 
  \draw[dashed] (v9) -- (v10); \draw[dashed] (v9) -- (v12) ; \draw[dashed] (v11) -- (v12) ; 
  \draw[dashed] (v14) -- (v16); \draw[dashed] (v15) -- (v16) ; 
  \end{scope}
  
\end{tikzpicture}
\end{center}
\caption{An illustration of $\pi_1$ (left) and $\pi_2$ (right) for  parameters $m=4$, $k=2$, and $n=m^k=16$ (while $m\neq 2^{k/2}$, this choice  illustrates well our process). The 4 colors represent the level-2 slices. The edges on the left are the edges of the Hasse diagrams of level-1 slices. The edges on the right are the edges of the Hasse diagrams of level-1 slices (dashed edges) and level-2 slices (solid edges).}
\label{fig:2}
\end{figure}

\begin{lemma}
Let $I\subset [n]$. The probability that $I$ is an independent set in $G$ is at most 
$$2^{-|I|k}\left(\frac{|I|m}{n}+1\right)^{kn/m}.$$
\end{lemma}

\begin{proof}
We want to argue that the events $\{I\cap K\text{ is an independent set in the Hasse diagram of }\pi|_K\}$ are independent for the different slices $K$. As the slices themselves are dependent, we need to be somewhat careful with this statement, so we present a formal argument. 

By the previous, if $I$ is an independent set, then for every slice $K$, $I\cap K$ is an independent set in the Hasse diagram of $\pi|_{K}$. We denote by $H_K$ the Hasse diagram of $\pi|_{K}$. For $j=1,\dots,k$, let $\mathcal{A}_j$ be the event that $I\cap K$ is an independent set in $H_K$ for every level-$j$ slice $K$. Then
$$\mathbb{P}(I\text{ is an independent set in }G)\leq \mathbb{P}(\mathcal{A}_1\cap\dots\cap\mathcal{A}_k)=\prod_{j=1}^{k}\mathbb{P}(\mathcal{A}_j|\mathcal{A}_1,\dots,\mathcal{A}_{j-1}).$$
Here, the events $\mathcal{A}_1,\dots,\mathcal{A}_{j-1}$ are fully determined by $\pi_{j-1}$. Fixing any outcome $\tau$ of $\pi_{j-1}$, the level-$j$ slices are also fully determined, let $L_{a,s}^{(j)}$ denote the corresponding outcome of the slice $K_{a,s}^{(j)}$. Then we have
\begin{align*}
    \mathbb{P}(\mathcal{A}_j|\pi_{j-1}=\tau)&=\prod_{L} \mathbb{P}(I\cap L\text{ is an independent set in }H_{L})\\
    &=\prod_{L}  2^{-|I\cap L|}(|I\cap L|+1)\\
    &=2^{-|I|}\prod_{L}(|I\cap L|+1),
\end{align*}
where the product ranges over all level-$j$ slices $L=L_{a,s}^{(j)}$. The first equality holds by the independence of the permutations $\beta^{(j)}_{a,s}$, the second equality holds by \Cref{lemma:gadget}, and the third equality holds by noting that the level-$j$ slices partition $[n]$. We further bound the right-hand-side by the AM-GM inequality. As the number of level-$j$ slices is $n/m$, we can write
$$\prod_{L}(|I\cap L|+1)\leq \left(\frac{\sum_{L} |I\cap L|+1}{n/m}\right)^{\frac{n}{m}}=\left(\frac{|I|m}{n}+1\right)^{\frac{n}{m}}.$$
Thus, filtering the event $\mathcal{A}_1\cap\dots\cap\mathcal{A}_{j-1}$ by the possible outcomes $\tau$ of $\pi_{j-1}$, we get
$$\mathbb{P}(\mathcal{A}_j|\mathcal{A}_1,\dots,\mathcal{A}_{j-1})\leq 2^{-|I|} \left(\frac{|I|m}{n}+1\right)^{\frac{n}{m}}.$$
In conclusion,
$$\mathbb{P}(I\text{ is an independent set in }G)\leq\prod_{j=1}^{k}\mathbb{P}(\mathcal{A}_j|\mathcal{A}_1,\dots,\mathcal{A}_{j-1})\leq 2^{-k|I|} \left(\frac{|I|m}{n}+1\right)^{\frac{kn}{m}}.$$
\end{proof}

In order to bound the independence number of $G$, we simply apply the union bound. Let $t\geq \frac{6n}{m}$, and write $x:=\frac{tm}{n}\geq 6$. We note that $\frac{1}{2}\geq  \frac{\log_2 (x+1)}{x}$. For any $t$ element set $I\subset [n]$, the previous lemma implies 
\begin{align*}
\mathbb{P}(I\text{ is an independent set in }G)&\leq 2^{-kt}\left(\frac{tm}{n}+1\right)^{\frac{kn}{m}}=2^{-kt}(x+1)^{\frac{kt}{x}}\\
&=\exp_2\left(kt\left(-1+\frac{\log_2(x+1)}{x}\right)\right)\leq 2^{-kt/2}.
\end{align*}
Thus, we get
$$\mathbb{P}(\alpha(G)\geq t)\leq \binom{n}{t}2^{-\frac{kt}{2}}< \left(\frac{en}{t}\right)^t2^{-\frac{kt}{2}}=\left(\frac{en}{2^{k/2} t}\right)^t\leq \left(\frac{em}{6\cdot 2^{k/2}}\right)^t<2^{-t}.$$
This proves that with high probability, the Hasse diagram of $\pi$ has independence number at most $$\frac{6n}{m}=6n2^{-\sqrt{\frac{1}{2}\log_2 n}},$$ finishing the proof of \Cref{thm:main}.

\section*{AI declaration}
\vspace{-10pt}
All mathematical content is due to the author. ChatGPT 5.6 Pro was used for proof checking and improving writing.

\section*{Acknowledgments}
\vspace{-10pt}
IT acknowledges the support of the Swedish Research Council grant VR 2023-03375.

\end{document}